\documentclass[graybox]{svmult}

\usepackage{mathptmx}       
\usepackage{helvet}         
\usepackage{courier}        
\usepackage{type1cm}        
\usepackage{makeidx}         
\usepackage{graphicx}        
\usepackage{multicol}        
\usepackage[bottom]{footmisc}

\usepackage{amssymb}
\usepackage{latexsym}
\usepackage{amsmath}
\usepackage{amsfonts}
\usepackage{amstext}
\usepackage{amscd}
\usepackage{indentfirst}

\def\N{{\mathbb N}}

\def\C{{\mathbb C}}
\def\1{{\bf 1}}

\newtheorem{thm}{Theorem}[section]

\newtheorem{prop}[thm]{Proposition}
\newtheorem{assump}[thm]{Assumption}

\newtheorem{defi}[thm]{Definition}

\newcommand{\dom}{\mathrm{dom}}

\newcommand{\cl}{\mathcal}

\newcommand{\R}{\mathbb{R}}
\newcommand{\slim}{\,\mbox{\rm s-}\hspace{-2pt} \lim}

\newcommand{\vertiii}[1]{{\left\vert\kern-0.25ex\left\vert\kern-0.25ex\left\vert #1
    \right\vert\kern-0.25ex\right\vert\kern-0.25ex\right\vert}}
\DeclareMathOperator*{\esssup}{ess\,sup}

\newcommand\gotH{{\mathfrak{H}}}

\newcommand\gotX{{\mathfrak{X}}}

\newcommand\dR{{\mathbb{R}}}

\newcommand\dN{{\mathbb{N}}}

\newcommand{\ga}{{\alpha}}
\newcommand{\gb}{{\beta}}
\newcommand{\gd}{{\delta}}
\newcommand{\gD}{{\Delta}}

\newcommand{\go}{{\omega}}

\newcommand\gt{{\tau}}

\newcommand{\gT}{{\Theta}}

\newcommand\cA{{\mathcal{A}}}
\newcommand\cB{{\mathcal{B}}}
\newcommand\cC{{\mathcal{C}}}

\newcommand\cK{{\mathcal{K}}}

\newcommand\cU{{\mathcal{U}}}

\newcommand{\ba}{\begin{array}}
\newcommand{\ea}{\end{array}}
\newcommand{\bea}{\begin{eqnarray}}
\newcommand{\eea}{\end{eqnarray}}
\newcommand{\bead}{\begin{eqnarray*}}
\newcommand{\eead}{\end{eqnarray*}}
\newcommand{\be}{\begin{equation}}
\newcommand{\ee}{\end{equation}}
\newcommand{\bed}{\begin{displaymath}}
\newcommand{\eed}{\end{displaymath}}
\newcommand{\bs}{\begin{split}}
\newcommand{\ees}{\end{split}}

\newcommand{\RE}{{\Re}e}
\newcommand{\wt}{\widetilde}

\usepackage{color}

\makeindex             

\begin{document}

\title*{{Operator-norm convergence of the Trotter product formula on Hilbert and
Banach spaces: a short survey}}
\titlerunning{Survey on the Trotter product formula}
\author{Hagen Neidhardt, Artur Stephan and Valentin A. Zagrebnov}
\authorrunning{H.~Neidhardt, A.~Stephan, V:A.~Zagrebnov}
\institute{H.~Neidhardt \at WIAS Berlin, Mohrenstr. 39, D-10117 Berlin, Germany\\
\email{hagen.neidhardt@wias-berlin.de}
\and
A.~Stephan\at HU Berlin, Institut f\"ur Mathematik, Unter den Linden 6, D-10099 Berlin, Germany\\ \email{stephan@math.hu-berlin.de}
\and
V.A.~Zagrebnov \at Universit\'{e} d'Aix-Marseille -
Institut de Math\'{e}matiques de Marseille  (UMR 7373), CMI-Technop\^{o}le Ch\^{a}teau-Gombert,
39, rue F. Joliot Curie, 13453 Marseille\\
\email{valentin.zagrebnov@univ-amu.fr}
}
\maketitle

\vspace{-2.5cm}

\textit{Dedicated to Haim Brezis and Louis Nirenberg un deep admiration}

\vspace{0.5cm}

\abstract{We give a review of results on the operator-norm convergence of the Trotter product formula
on Hilbert and Banach spaces, which is focused on the problem of its convergence rates. Some recent results concerning evolution semigroups are presented in details.
}

\section{Introduction} \label{sec:1}
\renewcommand{\theequation}{\arabic{section}.\arabic{equation}}

Recall that the product formula
\be\label{eq:1.1}
e^{-\gt C} = \lim_{n\to\infty}\left(e^{-\gt A/n}e^{-\gt B/n}\right)^n, \quad \gt \ge 0,
\ee
was established by S.~Lie (in 1875) for matrices where $C := A + B$. {{The proof of formula
\eqref{eq:1.1} can be carried over easily to bounded operators on Banach spaces}}. Moreover, a straightforward computation shows that the {{convergence rate is $O(1/n)$}}, i.e.
\be\label{eq:1.2}
\sup_{\gt \in [0,T]}\|e^{-\gt A/n}e^{-\gt B/n} - e^{-\gt C/n}\| = O(1/n).
\ee
H.~Trotter \cite{Trotter1959} has extended this result to unbounded operators $A$ and $B$ on Banach spaces,
but in the strong operator topology. He proved that if $A$ and $B$ are generators of contractions semigroups
on a separable Banach space such that the algebraic sum $A+B$ is a densely defined closable operator and the closure
$C = \overline{A+B}$ is a generator of a contraction semigroup, then
\be\label{eq:1.3}
e^{-\gt C} = \slim_{n\to\infty}\big(e^{-\gt A/n}e^{-\gt B/n}\big)^n \ ,
\ee
uniformly in $\gt \in [0,T]$ for any $T > 0$.

{{Formula \eqref{eq:1.3} is often
called the Trotter or the Lie-Trotter product formula. It was a long-time belief that this formula
is valid only in the \textit{strong} operator topology. But in nineties it was discovered that
under certain quite standard assumptions the strong convergence of the Trotter product formula can be improved to the \textit{operator-norm} convergence. In the following we give a review of these results.}}

The paper is organised as follows. In Section \ref{sec:2.1} we give an overview on operator-norm convergence of the Trotter product formula if the generators $A$ and $B$ are non-negative \textit{self-adjoint} operators. Section {{\ref{sec:2.2} summarises the case  when}} one of the generator is only a \textit{maximal accretive} operator.
{{Section \ref{sec:2.3} is devoted to the \textit{evolution} case}}, which arises in the theory of the abstract \textit{non-autonomous} Cauchy problem. {{These results are commented}} in Section \ref{sec:2.4}.

Section \ref{sec:3} {{is concerned with}} the operator-norm convergence of the Trotter product formula on
the Banach spaces. Section \ref{sec:3.1} presents the results under the assumption that one of generators
is for holomorphic semigroup. Section \ref{sec:3.2} considers again the evolution case {{but on Banach spaces.}} The relation between \textit{evolution semigroups} and propagators is explained in Section \ref{sec:3.3}. We comment the results in Section \ref{sec:3.4}.

In Section \ref{sec:4} we collect some examples and {{counterexamples}}. They show
what is expectable and what is {{not}} and even surprising.

{{We use below the following notations and definitions.}}
\begin{enumerate}

\item We use a definition of the semigroup \textit{generator} $C$ (\ref{eq:1.3}), which differs from the
standard one by a \textit{minus}, as it is in the book \cite{Kato1980}.

\item Furthermore, we widely use the so-called \textit{Landau symbols}:
\begin{align}
g(n) &= O(f(n)) \Longleftrightarrow \limsup_{n\to\infty} \left|\frac{g(n)}{f(n)}\right| < \infty \ ,
\nonumber \\
g(n) &= o(f(n)) \Longleftrightarrow \limsup_{n\to\infty} \left|\frac{g(n)}{f(n)}\right| = 0 \ ,
\nonumber \\
g(n) &= \gT(f(n)) \Longleftrightarrow 0 < \liminf_{n\to\infty} \left|\frac{g(n)}{f(n)}\right|
\le \limsup_{n\to\infty} \left|\frac{g(n)}{f(n)}\right| < \infty \ , \nonumber \\
g(n) &=  \go(f(n)) \Longleftrightarrow \limsup_{n\to\infty} \left|\frac{g(n)}{f(n)}\right| = \infty \ .
\nonumber
\end{align}

\item
{{We use the notation $C^{0,\beta}([0,T])$ for the H\"older  ($\gb \in (0,1)$) and, respectively, for
the Lipschitz ($\gb = 1$) continuous functions.}}

\end{enumerate}

\section{Trotter product formula on Hilbert spaces} \label{sec:2}

\subsection{Self-adjoint case}\label{sec:2.1}

Considering the Trotter product formula on a separable Hilbert space $\gotH$
T.~Kato has shown in \cite{Kato1974,Kato1978} that for
non-negative operators $A$ and $B$ the Trotter formula \eqref{eq:1.3} holds in the \textit{strong} operator
topology if $\dom(\sqrt{A}) \cap \dom(\sqrt{B})$ is dense in the Hilbert space and $C = A \dot+ B$
is the form-sum of operators $A$ and $B$.

Naturally the problem arises whether {{Kato's result}} can be extended
to the operator-norm convergence. A first attempt in this direction was undertaken by Rogava
\cite{Rogava1991}. He claimed that if $A$ and $B$ {{are}} non-negative self-adjoint operators such that
$\dom(A) \subseteq \dom(B)$ and the operator-sum: $C = A + B$, is \textit{self-adjoint}, then
\be\label{eq:1.4a}
\|(e^{-\gt A/n}e^{-\gt B/n})^n - e^{-\gt C}\| = O(\ln(n)/\sqrt{n}), \quad n\to\infty,
\ee
holds. In \cite{NeidhardtZagrebnov1998} it was shown that if one substitutes in above conditions the self-adjointness of the operator-sum by the $A$-smallness of $B$ with a relative bound less then \textit{one}, then \eqref{eq:1.4a} is true with the rate of convergence improved to
\bed
\|(e^{-\gt A/n}e^{-\gt B/n})^n - e^{-\gt C}\| = O(\ln(n)/n), \quad n\to\infty.
\eed

The problem in its original formulation was finally solved
in \cite{ITTZ2001}. There it was shown that the best possible in this general setup rate \eqref{eq:1.2} holds if the operator sum: $C = A+B$, is already a self-adjoint operator. Obviously, Rogava's result,
as well as many other results (including \cite{NeidhardtZagrebnov1998}), when the operator sum of generators is self-adjoint, follow from \cite{ITTZ2001} .

A new direction comes due to results for the fractional-power conditions. In \cite{NeiZag1999}, with  elucidation in \cite{IchNeiZag2004}, it was proven that assuming
\be\label{eq:2.6}
\dom(C^\ga) \subseteq \dom(A^\ga) \cap
\dom(B^\ga), \quad \ga \in ({1}/{2},1),\quad C = A \dot+B,
\ee
and
\be\label{eq:2.7}
\dom(A^{1/2}) \subseteq \dom(B^{1/2})
\ee
one obtains that
\bed
\sup_{\gt \in [0,T]}\|(e^{-\gt A/n}e^{-\gt B/n})^n - e^{-\gt C}\| = O(n^{-(2\ga-1)}).
\eed
Notice that formally $\ga = 1$ yields the rate obtained in \cite{ITTZ2001}.

We remark also that the results of \cite{IchNeiZag2004,NeiZag1999} do \textit{not} cover
the case $\ga = {1}/{2}$. Although, it turns out that in this case the Trotter
product formula converges on the operator norm:
\bed
\sup_{\gt \in [0,T]}\|(e^{-\gt A/n}e^{-\gt B/n})^n - e^{-\gt C}\| = o(1) \ ,
\eed
if $\sqrt{B}$ is \textit{relatively compact} with respect to $\sqrt{A}$, i.e.
$\sqrt{B}(I + A)^{-1/2}$ is compact, see \cite{NeiZag1999-b}.

\subsection{Nonself-adjoint case}\label{sec:2.2}

Another direction was related with attempts to extend the the Trotter, and the Trotter-Kato,
product formulae to the case of nonself-adjoint \textit{sectorial} generators {{\cite{CachZag2001}}}.
Let $A$ be a non-negative self-adjoint operator and let $B$ be a maximal \textit{accretive}
($\RE(Bf,f) \ge 0$ for $f \in \dom(B)$) operator, such that
\bed
\dom(A) \subseteq \dom(B)
\quad \mbox{and} \quad
\dom(A) \subseteq \dom(B^*).
\eed
If $B$ is $A$-small with a relative bound less than one, then the rate estimate \eqref{eq:1.4a} holds,
for generator $C$, which is the well-defined maximal accretive operator-sum: $C = A+B $, see \cite{CacNeiZag2001}.

In \cite{CacNeiZag2002} this result was generalised as follows. Let $A$ be a non-negative self-adjoint operator and let $B$ be a maximal accretive operator such that
$\dom(A) \subseteq \dom(B)$ and $B$ is $A$-small with relative bound less than one. If the condition
\bed
\dom((C^*)^\ga) \subseteq \dom(A^\ga) \cap \dom((B^*)^\ga), \quad  C = A + B \ ,
\eed
is satisfied for some $\ga \in (0,1]$, then the norm-convergent Trotter product formula:
\bed
\sup_{\gt \in [0,T]}\|(e^{-\gt A/n}e^{-\gt B/n})^n - e^{-\gt C}\| = O(\ln(n)/n^\ga) \ ,
\eed
holds as $n \to \infty$.

In fact, more results are known about the operator-norm Trotter product formula convergence for nonself-adjoint semigroups, but \textit{without} the rate estimates, see \cite{CachZag1999}.

\subsection{Evolution case}\label{sec:2.3}

At the first glance a very different result about a Trotter-type  product formula was obtained in \cite{IchinoseTamura1998}.
The authors consider instead of the self-adjoint operator $B$ a family $\{B(t)\}_{t\in [0,T]}$ of
self-adjoint operators on the separable Hilbert space $\gotH$ such that the condition
\be\label{eq:2.13}
\dom(A^\ga) \subseteq \dom(B(t)), \quad t \in [0,T],
\ee
is satisfied for some $\ga \in [0,1)$ for $A$, which  is a non-negative self-adjoint operator.
Then the operator sum $C(t) = A + B(t)$ defines a family of self-adjoint operators in $\gotH$ space such that $\dom(C(t)) = \dom(A)$, $t \in [0,T]$. With the family $\{C(t)\}_{t\in [0,T]}$ one associates the \textit{evolution} equation:

\be\label{eq:2.14}
\partial_{t}u(t) = -C(t)u(t), \quad t\in [0,T] \ ,
\ee
corresponding to the non-autonomous Cauchy problem with initial condition $u_0  = u (0)$ for $t=0$.

It turns out that equation \eqref{eq:2.14} admits a \textit{propagator} $\{U(t,s)\}_{(t,s) \in \bar\gD}$,
$\bar\gD = \{(t,s) \in [0,T]\times[0,T]: 0 \le s \le t \le T\}$, which solves this problem. We remind that a family  $\{U(t,s)\}_{(t,s) \in \gD}$ of bounded operators
is called a propagator if the operator-valued function $U(\cdot,\cdot): \bar\gD \longrightarrow \cB(\gotH)$
is strongly continuous and verifies the conditions:
\begin{align}
U(t,t) &= I\quad  \mbox{for} \quad t \in [0,T],\label{eq:2.15}\\
U(t,s) &=U(t,r)U(r,s) \quad \mbox{for} \quad t,r,s \in [0,T] \quad  \mbox{with} \quad s \le r \le t \ .
\label{eq:2.16}
\end{align}
Let $\{t_j\}^N_{j= 0}$ be a partition of the closed interval ${{[0,t]}}$:
\bed
0 = t_0 < t_1 < \ldots < t_{N-1} < t_N = t, \quad t_j = \gt j, \gt = {t}/{N} \ ,
\eed
for any $0 < t \le T$. Further, let
\be\label{eq:2.18}
Q_j(t,s;n) = e^{-\tfrac{t-s}{n}A}e^{-\tfrac{t-s}{n}B(s + j\tfrac{t-s}{n})}, \quad j  = 0,1,\ldots,n \ ,
\ee
and
\be\label{eq:2.19}
\begin{split}
E(t,s;n) &:= Q_{n-1}(t,s;n)Q_{n-2}(t,s;n) \times \cdots \times Q_1(t,s;n)Q_0(t,s;n)\\
 &:= \prod_{j=0}^{(n-1)\leftarrow} Q_j(t,s\,;n) \ ,
\end{split}
\ee
where the symbol $\prod_{j=1}^{n\leftarrow}$ means that the product is increasingly ordered in $j$ from the right to the left.
If in addition to assumption \eqref{eq:2.13}, the condition
\be\label{eq:2.20}
\|A^{-\ga}(B(t) - B(s))A^{-\ga}\| \le L_1 |t-s|, \quad t,s \in [0,T], \quad L_1 > 0,
\ee
is satisfied, then in \cite{IchinoseTamura1998} it was proved that
the propagator $\{U(t,s)\}_{(t,s)\in \bar\gD}$,
which solves the Cauchy problem \eqref{eq:2.14}, admits the approximation
\be\label{eq:2.21}
\sup_{t\in [0,T]}\|U(t,0) - E_n(t,0;n)\| = O(\ln(n)/n) \quad \mbox{as \;\; $n\to\infty$}.
\ee
Scrutinising the proof in \cite{IchinoseTamura1998} one finds that in fact
the claim \eqref{eq:2.21} can be slightly generalised to any interval $\bar\gD$
\be\label{eq:2.22}
\sup_{(t,s)\in\bar\gD}\|U(t,s) - E_n(t,s;n)\| = O(\ln(n)/n) \quad \mbox{as \;\; $n\to\infty$}.
\ee

At the first glance, it seems that the result \eqref{eq:2.22} is quite far from the Trotter
product formula. However, this is not the case.
To show this we follow the evolution semigroup approach to evolution equations developed in \cite{NeiSteZag2016,NeiZag2009}.
Let us introduce the Hilbert space ${{L^2([0,T],\gotH)}}$ and consider on this space
the semigroup
\be\label{eq:2.23}
(\cU(\gt)f)(t) := U(t,t-\gt)\chi_{[0,T]}(t-\gt)f(t-\gt), \quad f \in L^2([0,T],\gotH),
\ee
$t \in [0,T]$. It turns out that $\{\cU(\gt)\}_{\gt \in \R_+}$ is a $C_0$-semigroup on the Hilbert space
of the $\gotH$-valued trajectories $L^2([0,T],\gotH)$. By $\cK$ we denote its generator, then $\cU(\gt) = e^{-\gt\cK}$, $\gt \ge 0$.

Let us introduce two multiplication operators
\be\label{eq:2.24}
\begin{split}
(\cA f)(t) &= Af(t), \quad f \in \dom(\cA),\\
\dom(\cA) &:= \left\{f \in L^2([0,T],\gotH):
\begin{matrix}
f(t) \in \dom(A) \quad \mbox{for a.e. $t \in [0,T]$}\\
Af(t) \in L^2([0,T],\gotH)
\end{matrix}
\right\} \ ,
\end{split}
\ee
and
\be\label{eq:2.25}
\begin{split}
(\cB f)(t) &= B(t)f(t), \quad f \in \dom(\cB),\\
\dom(\cB) &:= \left\{f \in L^2([0,T],\gotH):
\begin{matrix}
f(t) \in \dom(B(t)) \quad \mbox{for a.e. $t \in [0,T]$}\\
B(t)f(t) \in L^2([0,T],\gotH)
\end{matrix}
\right\} \ .
\end{split}
\ee
Note that both operators are well-defined and self-adjoint in ${{L^2([0,T],\gotH)}}$.
Moreover, condition (\ref{eq:2.13}) yields: $\dom(\cA) \subseteq \dom(\cB)$.

By $D_0$ we define in $L^2([0,T],\gotH)$ the generator of the {{right-shift}} semigroup $e^{-\gt D_0}$
given by
\bed
(e^{-\gt D_0}f)(t) = \chi_{[0,T]}(t-\gt)f(t-\gt), \quad f \in L^2([0,T],\gotH) .
\eed
Notice that the operator $D_0$ is defined by
\bed
\begin{split}
(D_0f)(t) &= \frac{\partial}{\partial t}f(t),\\
f \in \dom(D_0) &:= \{f \in W^{2,2}([0,T],\gotH): f(0) = 0\}.
\end{split}
\eed

Collecting these definitions we introduce the operator
\be\label{eq:2.27}
\begin{split}
(\wt\cK f) &= D_0f + \cA f + \cB f, \\
f \in \dom(\wt \cK) &:= \dom(D_0) \cap \dom(\cA) \cap \dom(\cB).
\end{split}
\ee
Since $\{U(t,s)\}_{(t,s)\in\bar\gD}$ is a propagator solving the evolution equation
\eqref{eq:2.14} one deduces that by virtue of assumptions
\eqref{eq:2.13} and \eqref{eq:2.20} the operator $\wt\cK$ is closable and that its closure
coincides with generator $\cK$, see \cite[Theorem 4.5]{NeiSteZag2016}.

Furthermore, let us define the operator
\bed
(\cK_0f)(t) = D_0f + \cA f, \quad f \in \dom(\cK_0) := \dom(D_0) \cap \dom(\cA).
\eed
Then operator $\cK_0$ is a generator of the $C_0$-semigroup, which has the form
\bed
(e^{-\gt \cK_0}f)(t) = e^{-\gt A}\chi_{t\in [0,T]}(t-\gt)f(t-\gt), \quad f \in L^2([0,T],\gotH).
\eed
Note that we obviously get that $\cK = \overline{\cK_0 + \cB} = {{\cK_0 + \cB}}$.

For the pair $\{\cK_0,\cB\}$ one obtains the following result.
\begin{prop}\label{prop:2.1}
Let $A$ be a non-negative self-adjoint operator on the separable Hilbert space $\gotH$ and let
$\{B(t)\}_{t\in[0,T]}$ be a family of non-negative self-adjoint operators.
If the assumptions \eqref{eq:2.13} and \eqref{eq:2.20} are satisfied,
then
\be\label{eq:2.31}
\sup_{\gt \ge 0}\|(e^{-\gt\cK_0/n}e^{-\gt\cB/n})^n - e^{-\gt\cK}\| = O(\ln(n)/n) \quad \mbox{for $n \to \infty$} \ ,
\ee
and
\be\label{eq:2.32}
\sup_{\gt \ge 0}\|(e^{-\gt\cB/n}e^{-\gt\cK_0/n})^n - e^{-\gt\cK}\| = O(\ln(n)/n) \quad \mbox{for $n \to \infty$} \ .
\ee
\end{prop}
\begin{proof}
A straightforward computation shows that for $f \in L^2([0,T],\gotH)$ one gets
\bed
((e^{-\gt\cK_0/n}e^{-\gt\cB/n})^nf)(t) = E(t,t-\gt\,;n)\chi_{[0,T]}(t-\gt)f(t-\gt),
\eed
where $E(t,s\,;n)$ is given by \eqref{eq:2.19}. Taking into account
\eqref{eq:2.23} we obtain
\bed
\begin{split}
((e^{-\gt\cK_0/n}&e^{-\gt\cB/n})^nf)(t) - (e^{-\gt \cK}f)(t)\\
& =
(E(t,t-\gt\,;n) - U(t,t-\gt))\chi_{[0,T]}(t-\gt)f(t-\gt) \ ,
\end{split}
\eed
which yields the estimate
\bed
\begin{split}
\|((e^{-\gt\cK_0/n}&e^{-\gt\cB/n})^nf)(t) - (e^{-\gt \cK}f)(t)\|\\
& \le
\chi_{[0,T]}(t-\gt)\|{{E}}(t,t-\gt\,;n) - U(t,t-\gt)\|\;\|f(t-\gt)\|.
\end{split}
\eed
Hence, this implies
\bed
\begin{split}
\sup_{\gt \ge 0} \|&(e^{-\gt\cK_0/n}e^{-\gt\cB/n})^n f - e^{-\gt \cK}f\|^2\\
&\le
\sup_{\gt \ge 0}\int^T_0 \chi_{[0,T]}(t-\gt)\|E(t,t-\gt\,;n) - U(t,t-\gt)\|^2\;\|f(t-\gt)\|^2dt\\
&\le
\int^T_0 \sup_{(t,s)\in\bar\gD}\|E(t,s\,;n) - U(t,s)\|^2\;\|f(r)\|^2dr.
\end{split}
\eed
Using \eqref{eq:2.22} we immediately obtain \eqref{eq:2.31}.
{{Similarly \eqref{eq:2.32} follows from \cite{IchinoseTamura1998}.}}
\hfill$\Box$
\end{proof}

\subsection{Comments}\label{sec:2.4}

\paragraph{Section \ref{sec:2.1}}
\vspace{-2ex}
{{The operator-norm convergence rate $O(1/n)$ of \cite{ITTZ2001}
for pairs of non-negative self-adjoint operators is \textit{sharp} and \textit{ultimate optimal} due to
observations in \cite{Tam2000}.
The same remark concerns the sharpness and optimality of the rate $O(1/n^{2\ga-1})$ obtained first
in \cite{NeiZag1999} under assumption \eqref{eq:2.6} together with the $A^{\alpha}$-smallness of $B^{\alpha}$ with relative bound less then one. Then the same ultimate sharp rate was proven in \cite{IchNeiZag2004}, when the smallness condition is relaxed to the mild subordination \eqref{eq:2.7}. It is an open problem whether the assumption \eqref{eq:2.7} is really necessary.}}

\vspace{-4ex}
\paragraph{Section \ref{sec:2.2}}
\vspace{-2ex}
It is unclear whether the convergence rates $O(\ln(n)/n)$ and
$O(\ln(n)/n^\ga)$ are sharp. One expects convergence rates identical to that in Section \ref{sec:2.1}.

\vspace{-4ex}
\paragraph{Section \ref{sec:2.3}}
\vspace{-2ex}
The approach used here was developed in \cite{NeiSteZag2016,Nei1981,Nei1981-b,Nei1982,NeiZag2009,Nickel1996}.
The idea is to transform a \textit{time-dependent} evolution problem to a \textit{time-independent} problem, see also the next section.

Let us add some remarks. One easily checks that the operator $\cK_0$ is not
self-adjoint whereas the operator $\cB$ is self-adjoint. However, $\cK_0$ is maximal accretive.
This is in some sense
in contrast to Section \ref{sec:2.2}, where the pair $\{A,B\}$ consists of a self-adjoint operator
$A$ and an maximal accretive operator $B$ such that $\dom(A) \subseteq \dom(B)$ and
$\dom(A) \subseteq \dom(B^*)$. In the evolution case the conditions
$\dom(\cK_0) \subseteq \dom(\cB)$ and $\dom(\cK^*_0) \subseteq \dom(\cB)$
are satisfied but in the {{reversed}} order with respect to Section \ref{sec:2.2}.
However, the convergence rate $O(\ln(n)/n)$ is not affected by this.

The proof of the estimate \eqref{eq:2.21} and \eqref{eq:2.22} are very involved.
Naturally the problem arises whether one can give a direct proof the estimate
\eqref{eq:2.21} avoiding those propagator estimates.

\section{Trotter product formula on Banach spaces}\label{sec:3}

\subsection{Holomorphic case}\label{sec:3.1}

There are only few generalisations of the results of Section \ref{sec:2} to Banach spaces.
The main obstacle for that is the fact that the concept of self-adjointness is missing in the
Banach spaces. One of solution is to relax the self-adjointness replacing the non-negative self-adjoint generator $A$ by a generator of the holomorphic semigroup. The following result was proved in \cite{CachZag2001-b}.
\begin{thm}[{\cite[Theorem 3.6 and Corollary 3.7]{CachZag2001-b}}]\label{th:3.1}
Let $A$ be a generator of a holomorphic contraction semigroup on the
separable Banach space $\gotX$ and let $B$ a generator of a contraction semigroup on $\gotX$.

\item[\;\;\rm (i)] If for some $\ga \in (0,1)$ the condition
\bed
\dom(A^\ga) \subseteq \dom(B) \ ,
\eed
holds and $\dom(A^*) \subseteq \dom(B^*)$ is satisfied, then
the operator sum $C = A + B$ is a generator of a contraction semigroup and
\be\label{eq:3.38}
\sup_{\gt \in [0,T]}\|(e^{-\gt B/n}e^{-\gt A/n})^n - e^{-\gt C}\| = O(\ln(n)/n^{1-\ga})\ ,
\ee
for any $T > 0$.

\item[\;\;\rm (ii)] If for some $\ga \in (0,1)$ the condition
\bed
\dom((A^\ga)^*) \subseteq \dom(B^*) \ ,
\eed
is satisfied and $\dom(A) \subseteq \dom(B)$ is valid, then $C = A+B$ is the generator of a contraction semigroup and
\be\label{eq:3.40}
\sup_{\gt \in [0,T]}\|(e^{-\gt A/n}e^{-\gt B/n})^n - e^{-\gt C}\| = O(\ln(n)/n^{1-\ga})\ ,
\ee
for any $T > 0$.
\end{thm}
\begin{thm}[{\cite[Theorem 3.6 and Corollary 3.7]{CachZag2001-b}}]\label{th:3.2}
Let $A$ be a generator of a holomorphic contraction semigroup on $\gotX$ and let $B$ a generator of a contraction semigroup on $\gotX$.
If $B$ is in addition a bounded operator, then
\bed
\sup_{\gt \in [0,T]}\|(e^{-\gt B/n}e^{-\gt A/n})^n - e^{-\gt C}\| = O((\ln(n))^2/n),
\eed
and
\bed
\sup_{\gt \in [0,T]}\|(e^{-\gt A/n}e^{-\gt B/n})^n - e^{-\gt C}\| = O((\ln(n))^2/n),
\eed
for any $T > 0$.
\end{thm}

\subsection{Evolution case}\label{sec:3.2}

Similarly of Section \ref{sec:2.3} let us consider a generator $A$
of a holomorphic semigroup on the separable Banach space $\gotX$ and a family of $\{B(t)\}_{t\in[0,T]}$ of generators of holomorphic semigroups on $\gotX$. We make the following
assumptions:
\begin{assump}\label{ass:1.1}
{\rm

 \item[\;\;(A1)] The operator $A$ is a generator of a  holomorphic contraction semigroup on $\gotX$
such that $0 \in \rho(A)$.

\item[\;\;(A2)] Let $\{B(t)\}_{t\in[0,T]}$ be a family of closed operators such that for a.e. $t\in [0,T]$ and some $\ga \in (0,1)$ the condition $\dom(A^\alpha)\subset\dom(B(t))$ is satisfied such that
 \begin{align*}
  C_\alpha:=\mathrm{ess ~sup}_{t\in [0,T]}\|B(t)A^{-\alpha}\|_{\cl B(X)}<\infty \ .
 \end{align*}

\item[\;\;(A3)] Let $\{B(t)\}_{t\in [0,T]}$ be a family of generators of contraction semigroups
in $\gotX$ such that the function
$[0,T] \ni t\mapsto (B(t)+\xi)^{-1}x\in \gotX$ is
strongly measurable for any $x\in \gotX$ and any $\xi> b > 0$.

 \item [\;\;(A4)] We assume that $\dom(A^*)\subset \dom(B(t)^*)$ and
 \begin{align*}
 C_1^*:={\rm ~ess~sup}_{t\in [0,T]}\|B(t)^*(A^*)^{-1}\|_{\cl B(X^*)}<\infty,
 \end{align*}
 where $A^*$ and  $B(t)^*$ denote operators which are adjoint of $A$ and $B(t)$, respectively.

 \item[\;\;(A5)] There exists $\gb \in (\ga,1)$ and a constant $L_\gb > 0$ such that
 for a.e. $t,s\in [0,T]$ one has the estimate:
 \begin{align*}
  \|A^{-1}(B(t)-B(s))A^{-\alpha}\|\leq {{L_\gb}} |t-s|^\gb \ .
 \end{align*}

\item[\;\;(A6)] There exists a constant ${{L_1 > 0}}$ such that
 for a.e. $t,s\in [0,T]$ one has the estimate:
 \begin{align*}
  \|A^{-\ga}(B(t)-B(s))A^{-\alpha}\|\leq {{L_1}}|t-s| \ .
 \end{align*}
}
\end{assump}

The assumption $0 \in \rho(A)$ in (A1) is made for simplicity.
We note that assumption (A3) is similar to \eqref{eq:2.13}. Assumption (A4) is automatically satisfied for self-adjoint operators. Assumption (A5) is a modification of \eqref{eq:2.20} while assumption (A6) coincides with \eqref{eq:2.20}.

With the family $\{C(t)\}_{t\in[0,T]}$,  $C(t) = A + B(t)$, one associates the evolution equation
\eqref{eq:2.14}. It turns out that under  the assumptions (A1) and (A2) the family $\{C(t)\}_{t\in[0,T]}$ consists of generators of {{contraction semigroups}}.

In accordance with Section \ref{sec:2.3} we consider the Banach space ${{L^p([0,T],\gotX)}}$
for some fixed $p \in [0,1)$ and introduce the multiplication operators $\cA$ and $\cB$, cf.
\eqref{eq:2.24} and \eqref{eq:2.25}. Under the assumptions (A1)-(A3) the operators $\cA$ and $\cB$
are generators of {{contraction}} semigroups such that $\dom(\cA) \subseteq \dom(\cB)$, in particular, $\cA$ is a generator of holomorphic semigroup. Similarly, one can introduce the multiplication operator $\cC$ induced by the family $\{C(t)\}_{t\in[0,T]}$ which is also a generator of a holomorphic semigroup.
Notice that $\cC = \cA + \cB$ and $\dom(\cC) = \dom(\cA)$.

Let $D_0$ the generator of
the right-shift semigroup on $L^p([0,T],\gotX)$, i.e.
\bed
(e^{-\gt D_0}f)(t) = f(t-\gt)\chi_{[0,T]}(t-\gt)f(t-\gt), \quad f \in L^p([0,T],\gotX),
\eed
cf. \eqref{eq:2.27}. We consider the operator
\bed
\begin{split}
\wt\cK f &= D_0 f + \cA f + \cB f,\\
f \in \dom(\wt\cK) &= \dom(D_0) \cap \dom(\cA) \cap \dom(\cB),
\end{split}
\eed
cf. \eqref{eq:2.27}. Assuming (A1)-(A3) it was shown in \cite{NeiSteZag2016}
that the operator $\wt K$ is closable
and its closure $\cK$ is the generator of a semigroup.

Furthermore, we set
\bed
\wt\cK_0 f = D_0 f + \cA f, \quad
f \in \dom(\wt\cK_0) = \dom(D_0) \cap \dom(\cA),
\eed
In contrast to the Hilbert space the operator $\wt\cK_0$ is not necessary a generator
of a semigroup. However, the operator $\cK_0$ closable and its closure $\cK_0$ is a generator.
Notice that
$\cK$ coincides with the algebraic sum of $\cK_0$ and $\cB$, i.e $\cK = \cK_0 + \cB$.

In \cite{NeiSteZag2016} the following theorem was proved.
\begin{thm}[{\cite[Theorem 7.8]{NeiSteZag2016}}]\label{th:3.4}
Let the assumptions  (A1)-(A4) be satisfied  for some $\ga \in (0,1)$. If (A5) holds, then
 \be\label{eq:3.46}
\sup_{\gt \ge 0}\| (e^{-\gt B/n}e^{-\gt \cK_0/n})^{n}-e^{-\tau\cK}\| = O(1/n^{\gb-\ga}),
\ee
when $n \to \infty$.
\end{thm}
Assuming instead of assumption (A5) the assumption (A6) the result slightly modifies, see
\cite{NeiSteZag2017}.
\begin{thm}[{\cite[Theorem 5.4]{NeiSteZag2017}}]\label{th:3.5}
Let the assumptions  (A1)-(A4) be satisfied  for some $\ga \in ({1}/{2},1)$.
If (A6) is valid, then for $n \to \infty$ one gets the asymptotic:
 \be\label{eq:3.47}
\sup_{\gt \ge 0}\| (e^{-\gt B/n}e^{-\gt \cK_0/n})^{n}-e^{-\tau\cK}\| = O(1/n^{1-\ga}).
\ee
\end{thm}
\subsection{Convergence rates for propagators}\label{sec:3.3}

The proof of both theorems does not use propagator approximations of type \eqref{eq:2.21} or
\eqref{eq:2.22}. However, Theorem \ref{th:3.4} and \ref{th:3.5} can be used to prove propagator approximations. To this end one has to introduce the notion of a evolution semigroup.
\begin{defi}\label{def:3.6}
A generator $\cK$ in $L^p([0,T],X)$, $p \in [1,\infty)$, is called a evolution generator if
 \item[\;\;\rm (i)] $\dom(\cK)\subset C([0,T],X)$ and $M(\phi)\dom(\cK)\subset\dom(\cK)$ for
 $\phi\in W^{1,\infty}([0,T])$,

 \item[\;\;\rm (ii)] $\cK M(\phi)f-M(\phi)\cK f=M(\dot{\phi})f$ for $f\in\dom(\cK)$ and
 $\phi\in W^{1,\infty}([0,T])$, where $\dot \phi=\partial_{t}\phi$,

 \item[\;\;\rm (iii)] the domain $\dom(\cK)$ has a dense cross-section, i.e.
for each $t \in (0,T]$ the set
\bed
[\dom(\cK)]_t := \{x \in \gotX: \exists f \in \dom(\cK) \;\; \mbox{such that} \;\; x \in f(t)\},
\eed
is dense in $\gotX$.
\end{defi}
By $M(\phi)$, $L^\infty([0,T])$, the bounded operator
\bed
(M(\phi)f)(t) = \phi(t)f(t), \quad f \in L^p([0,T],\gotX).
\eed
is meant.

One can check that the operator $\cK$ defined as the closure of $\wt \cK$ is an evolution generator, cf. \cite[Theorem 1.2]{NeiSteZag2016}. Evolution generators a directly related to propagators. For this purpose one has slightly weaken the notion of a propagator defined in Section
\ref{sec:2.3}.
\begin{defi}\label{def:3.7}
{\rm
Let $\{U(t,s)\}_{(t,s)\in \Delta}$, $\gD = \{(t,s) \in (0,T] \times (0,T]: s \le t \le T\}$,
be a strongly continuous family of bounded operators on $\gotX$. If the conditions
\begin{align}
&U(t,t)=I \quad \mbox{for} \quad t\in (0,T] \ ,\label{eq:3.49}\\
&U(t,r)U(r,s)=U(t,s) \quad \mbox{for} \quad t,r,s\in (0,T] \quad \mbox{with~} \quad s\leq r\leq t \ ,\label{eq:3.50}\\
&\|U\|_{\cB(X)} :=\sup_{(t,s)\in\Delta}\|U(t,s)\|<\infty\label{eq:3.51}
\end{align}
are satisfied, then $\{U(t,s)\}_{(t,s)\in \Delta}$ is called a propagator.
}
\end{defi}
Comparing with Section \ref{sec:2.3} we note that $\gD$ slightly differs from $\bar\gD$. Indeed,
$\gD \subseteq \bar \gD$ but $\overline{\gD} = \bar\gD$. Restricting \eqref{eq:2.15} and
\eqref{eq:2.16} to $(0,T]$ we get \eqref{eq:3.49} and \eqref{eq:3.50}, respectively. Condition
\eqref{eq:3.51} is necessary because the set $\gD$ is not closed.

It is known that there is an one-to-one correspondence between the set of all evolution generators on $L^p([0,T],\gotX)$ and the set of all propagators in the sense of Definition \ref{def:3.7} established by
\bed
(e^{-\gt \cK}f)(t) = U(t,t-\gt)\chi_{[0,T]}(t-\gt)f(t-\gt), \quad f \in L^p([0,T],\gotX),
\eed
cf. \cite[Theorem 3.3]{NeiSteZag2016} or \cite[Theorem 4.12]{Nei1981}.

Let $\cK_0$ be the generator of an evolution semigroup $\{\cU_0(\gt)\}_{\gt \ge 0}$ and let $\cB$
be a multiplication operator induced by a measurable family $\{B(t)\}_{t\in [0,T]}$ of generators of
contraction semigroups. Note that in this case the multiplication operator $\cB$ is a
generator of a contraction semigroup $(e^{- \tau \, \cB} f)(t) = e^{- \tau \, B(t)} f(t)$,
on the Banach space $L^p([0,T],X)$. Since $\{\cU_0(\gt)\}_{\gt \ge 0}$ is an evolution semigroup, then there is a propagator $\{U_0(t,s)\}_{(t,s) \in \gD}$ such that the representation
\bed
(\cU_0(\gt)f)(t) = U_0(t,t-\gt)\chi_{[0,T]}(t-\gt)f(t-\gt), \quad f \in L^p([0,T],X),
\eed
is valid for a.e. $t \in [0,T]$ and $\gt \ge 0$. Then we define
\bed
Q_j(t,s;n) := U_0(s + j\tfrac{(t-s)}{n},s+ (j-1)\tfrac{(t-s)}{n})
e^{-\frac{(t-s)}{n} B\big(s + (j-1)\tfrac{(t-s)}{n}\big)}
\eed
where $j \in \{1,2,\ldots,n\}$, $n \in \dN$, $(t,s) \in \gD$, and we set
\bed
V_n(t,s) := \prod^{n\,\leftarrow}_{j=1}Q_j(t,s;n), \quad n \in \dN, \quad (t,s) \in \gD,
\eed
where the product is increasingly ordered in $j$ from the right to the left.
Then a straightforward computation shows that the representation
\bed
\left(\left(e^{-\gt \cK_0/n}e^{-\gt \cB/n}\right)^n f\right)(t) =
V_n(t,t-\gt)\chi_{[0,T]}(t-\gt)f(t-\gt) \ ,
\eed
$f \in L^p([0,T],X)$, holds for each $\gt \ge 0$ and a.e. $t \in [0,T]$. Similarly we can introduce
\bed
G_j(t,s;n) = e^{-\tfrac{t-s}{n}B(s + j\tfrac{t-s}{n})}
U_0(s + j\tfrac{t-s}{n},s + (j-1)\tfrac{t-s}{n})
\eed
where $j \in \{1,2,\ldots,n\}$, $n \in \dN$, $(t,s) \in \gD$. Let
\bed
U_n(t,s) := \prod^{n\,\leftarrow}_{j=1}G_j(t,s;n), \quad n \in \dN, \quad (t,s) \in \gD,
\eed
where the product is again increasingly ordered in $j$ from the right to the left. We verify that
\bed
\left(\left(e^{-\gt \cB/n}e^{-\gt \cK_0/n}\right)^n f\right)(t) =
U_n(t,t-\gt)\chi_{[0,T]}(t-\gt)f(t-\gt) \ ,
\eed
$f \in L^p([0,T],X)$, holds for each $\gt \ge 0$ and a.e. $t \in [0,T]$.
\begin{prop}[{\cite[Proposition 2.1]{NeiSteZag2017b}}]\label{prop:3.8}
Let $\cK$ and $\cK_0$ be generators of evolution semigroups on the Banach space $L^p([0,T],X)$ for some
$p \in [1,\infty)$. Further, let $\{B(t))\}_{t\in [0,T]}$ be a strongly measurable family of
generators of contraction on $\gotX$.
Then
\bed
\sup_{\gt\ge 0}\left\|e^{-\gt \cK} - \left(e^{-\gt \cK_0/n}e^{-\gt \cB/n}\right)^n\right\|
= \esssup_{(t,s)\in \gD}\|U(t,s) - V_n(t,s)\|, \quad n\in \dN,
\eed
and
\bed
\sup_{\gt\ge 0}\left\|e^{-\gt \cK} - \left(e^{-\gt \cB/n}e^{-\gt \cK_0/n}\right)^n\right\|
= \esssup_{(t,s)\in \gD}\|U(t,s) - U_n(t,s)\|, \quad n\in \dN.
\eed
\end{prop}
Let us introduce the approximations
\bed
\begin{split}
 U_n(t,s)  &:= \prod_{j=1}^{n\leftarrow} G_j(t,s\,;n), \quad n = 1,2,\ldots\:,\\
G_j(t,s\,;n)   &:= e^{-\frac{t-s} n B(s+ j\frac{t-s} n)}e^{-\frac {t-s} n A},\quad j = 0,1,2,\ldots,n,
\end{split}
\eed
$(t,s) \in \gD$, with increasingly ordered product in $j$ from the right to the left.
{{From Theorem \ref{th:3.4} and Proposition \ref{prop:3.8} one immediately obtains the following theorem.}}
\begin{thm}[{\cite[Theorem 1.4]{NeiSteZag2016}}]
Let the assumptions  (A1)-(A4) be satisfied. If (A5) holds, then for $n \to \infty$
the rate:
 \be\label{eq:1.8}
 \esssup_{(t,s)\in\gD}\|U_n(t,s)-U(t,s)\| = O(1/n^{\gb-\ga}) \ .
\ee
\end{thm}
{{From Theorem \ref{th:3.5} and Proposition \ref{prop:3.8}}} we get
\begin{thm}[{\cite[Theorem 5.6]{NeiSteZag2017}}]
Let the assumptions  (A1)-(A4) be satisfied  for some $\ga \in ({1}/{2},1)$.
If (A6) is valid, then for $n \to \infty$ one obtains a better rate:
 \bed
 \esssup_{(t,s)\in\gD}\|U_n(t,s)-U(t,s)\| = O(1/n^{1-\ga}) \ .
\eed
\end{thm}

\subsection{Comments}\label{sec:3.4}

\paragraph{Section \ref{sec:3.1}}
\vspace{-2ex}
It is unclear whether Theorem \ref{th:3.1} is sharp. Theorem \ref{th:3.1} should be valid
if the contractivity of the involved semigroups is replaced Trotter-stability.
A pair of generators $\{A,B\}$ is called Trotter-stable if the Trotter product is uniformly bounded in $n \in \N$, i.e.,  if
\bed
\sup_{n\in\N}\sup_{\gt \ge 0}\left\|\left(e^{-\gt A/n}e^{-\gt B/n}\right)^n\right\| < \infty.
\eed
It turns out that the pair $\{A,B\}$ is Trotter-stable if and only $\{B,A\}$ is Trotter-stable.

\vspace{-4ex}
\paragraph{Section \ref{sec:3.2}}
\vspace{-2ex}
Theorems \ref{th:3.4} and {{\ref{th:3.5}}} are proved in \cite{NeiSteZag2016} and
\cite{NeiSteZag2017} under the assumption that the pair $\{\cK_0,\cB\}$ is Trotter-stable.
The convergence rates \eqref{eq:3.46} and \eqref{eq:3.47} differ significantly from the convergence rate $O(\ln(n)/n)$ of Proposition \ref{prop:2.1}. It is an open problem whether the convergences rates \eqref{eq:3.46} and \eqref{eq:3.47} can be improved to $O(\ln(n)/n)$.
One has to mention that the convergence \eqref{eq:3.47} coincides with that one of
\eqref{eq:3.40} despite the fact that $\cK_0$ is not a generator of a holomorphic  semigroup. Indeed, the generator $\cK_0$ is not holomorphic since $	e^{-\gt\cK_0} = 0$ for $\gt \ge T$.

\vspace{-4ex}
\paragraph{{{Section \ref{sec:3.3}}}}
\vspace{-2ex}
{{It is a bit surprising that the operator-norm convergence of the Trotter product formula
for the pairs: generator of evolution semigroup and multiplication operator, is equivalent to the operator-norm convergence of a certain approximation of the corresponding propagator, see
Proposition \ref{prop:3.8}.
In particular, this yields that two convergences:  \eqref{eq:2.22} and \eqref{eq:2.31}, are equivalent.}}

\section{Sharpness}\label{sec:4}

\subsection{Example}\label{sec:4.1}

Let us consider a ``solvable'' example.
We study bounded perturbations of the evolution generator  $D_0$.
To do this aim we consider $\gotX = \C$ and we denote by  $L^2([0,1])$ the Hilbert
space $L^2([0,T],\C)$.

For $t \in [0,1]$, let $q: t \mapsto q(t) \in L^\infty([0,1])$.
Then, $q$ induces a bounded multiplication operator $Q$ in the Banach space ${{L^2([0,1])}}$:
\bed
(Qf)(t) = q(t) f(t), \quad f\in {{L^2([0,1])}}.
\eed
For simplicity we assume that $q\geq 0$.
Then $Q$ generates on $L^p([0,1])$ a contraction semigroup $\{e^{- \tau Q}\}_{\tau \geq 0}$.
Since generator $Q$ is bounded, the closed operator $\cK:= D_0 + Q$, with domain
$\dom(\cK) = \dom(D_0)$,
is generator of a semigroup on $L^p([0,1])$. From \cite{Trotter1959} one gets immediately
\bed
 \slim_{n\to\infty}\left(e^{-\gt D_0/n}e^{-\gt Q/n}\right)^n = e^{-\gt (D_0+Q)}
\eed
uniformly in $\tau\in[0,T]$ for any $T > 0$. One easily checks that $\cK$
is an evolution generator. A straightforward computation shows that
\bed
\left(e^{-\tau(D_0 + Q)}f\right)(t)
= e^{-\int_{t-\tau}^t q(y) dy}\chi_{[0,1]}(t-\tau)f(t-\tau)
\eed
which yields that the propagator corresponding to $\cK$ is given by
\bed
U(t, s) = e^{-\int_s^t  q(y)dy}, \quad (t,s) \in \gD.
\eed
A simple computation shows that
\bed
\left(\left(e^{- \tau D_0/n}e^{- \tau Q/n}  \right)^n f\right)(t) =:
V_n(t,t-\gt)\chi_{[0,T]}(t-\gt)f(t-\gt) \ .
\eed
Then by straightforward calculations one finds that
\bed
V_n(t,s)= e^{-\tfrac{t-s}{n} \sum_{k=0}^{n-1} q(s + k\tfrac{t-s}{n})},\quad (t,s) \in \gD \ .
\eed
\begin{prop}[{\cite[Proposition 3.1]{NeiSteZag2017b}}]\label{prop:3.1}
Let $q \in L^\infty([0,T])$ be non-negative. Then
\bed
\begin{split}
\sup_{\gt \ge 0}&\left\|e^{-\gt(D_0 + Q)} - \left(e^{-\gt D_0/n}e^{-\gt Q/n}\right)^n\right\|_{\cB(L^p([0,1]))}
\\
&=
\gT\left(\esssup_{(t,s)\in\gD}\Big|\int^t_s q(y)dy - \frac{t-s}{n} \sum_{k=0}^{n-1} q(s +
k\tfrac{t-s}{n})\Big|\right)
\end{split}
\eed
as $n\to\infty$, where $\gT$ is the Landau symbol defined in the introduction.
\end{prop}

Note that by Proposition \ref{prop:3.8} the operator-norm
convergence rate of the Trotter product formula for the pair $\{D_0 , Q\}$ coincides with the convergence
rate of the integral Darboux-Riemann sum approximation of the Lebesgue integral.

\subsection{Results}\label{sec:4.2}

Below we give a series of examples which show the dependence of the convergence rate on the smoothness of the function $q \in L^\infty([0,T])$. First we consider the H\"older {{and Lipschitz}} continuous cases.
\begin{thm}[{\cite[Theorem 3.2]{NeiSteZag2017b}}]\label{th:4.2}
If the function: $q \in C^{0,\gb}([0,T])$, $\gb \in (0,1]$, is non-negative, then
for $n \to \infty$ one gets
\bed
\sup_{\gt \ge 0}\left\|e^{-\gt(D_0+Q)} - \left(e^{-\gt D_0/n}e^{-\gt Q/n}\right)^n\right\| =
O({1}/{n^\gb}) \ .
\eed
\end{thm}

Now a natural question that one may to ask is: what happens, when $q$ is simply \textit{continuous}?
\begin{thm}[{\cite[Theorem 3.3]{NeiSteZag2017b}}]\label{th:4.3}
If $q: [0,1] \rightarrow \C$, is continuous and non-negative, then for $n\to \infty$
\be\label{eq:4.73}
\left\|e^{-\gt(D_0+Q)} - \left(e^{-\gt D_0/n}e^{-\gt Q/n}\right)^n\right\| = o(1) \ .
\ee
\end{thm}

We comment that for a general continuous $q$ one can say nothing about the Trotter product formula convergence rate. Indeed, as it follows from the next theorem the convergence to zero in (\ref{eq:4.73}) may be \textit{arbitrary} slow.
\begin{thm}[{\cite[Theorem 3.4]{NeiSteZag2017b}}]\label{th:4.4}
 Let $\gd_n>0$ be a sequence with $\gd_n \to 0$ as $n \to \infty$. Then there exists a continuous
 function $q:[0,1] \rightarrow \dR$ such that
 \bed
\sup_{\gt \ge 0}\left\|e^{-\gt(D_0 + Q)} - \left(e^{-\gt D_0/n}
e^{-\gt Q/n}\right)^n\right\|_{\cl B( L^p([0,1]))} = \go(\gd_n) \ ,
 \eed
as $n\to\infty$, where $\omega$ is the Landau symbol defined in the Introduction.
\end{thm}

Our final comment concerns the case when $q$ is only \textit{measurable}. Then it can happen that the
Trotter product formula for that pair $\{D_0 , Q\}$ does \textit{not} converge in the
\textit{operator-norm} topology:

\begin{thm}[{\cite[Theorem 3.5]{NeiSteZag2017b}}]\label{th:4.5}
 There is a non-negative measurable function $q \in L^\infty([0,1])$, such that
 \bed
 \liminf_{n\to\infty}\; \sup_{\gt \ge 0}\left\|e^{-\gt(D_0 + Q)} - \left(e^{-\gt D_0/n}
 e^{-\gt Q/n}\right)^n\right\|_{\cl B( L^p([0,1]))} > 0 \ .
 \eed
\end{thm}

We note that Theorem \ref{th:4.5} does not exclude the convergence of the Trotter product formula
for the pair $\{D_0 , Q\}$ in the \textit{strong} operator topology.

\subsection{Comments}\label{sec:4.3}

\paragraph{Section \ref{sec:4.1}}
\vspace{-2ex}

Our example can be considered as a kind of solvable model.
It fits into the evolution cases considered in Sections \ref{sec:2.3} and
\ref{sec:3.2}. Indeed, one has to set $\gotX = \C$, $p=2$, $L^p([0,T],\gotX) = L^2([0,T])$,
$A = I$ and $B(t) = q(t)$ and $\cB = Q$. One easily checks that
\bed
(e^{-\gt/n \cK_0}e^{-\gt \cB/n})^n = e^{-\gt}(e^{-\gt/n D_0}e^{-\gt Q/n})^n
\eed
which shows that the convergence rate of $(e^{-\gt/n D_0}e^{-\gt Q/n})^n$ coincides
with that one $(e^{-\gt/n \cK_0}e^{-\gt \cB/n})^n$. One easily checks that
the choice $A = I$, $B(t) = q(t)$, $q(t) \ge 0$, guarantees  the assumptions (A1)-(A4)
If $q \in C^{0,\gb}([0,T])$, then the assumption (A5) is satisfied. If $Q$ is Lipschitz continuous,
i.e. $q \in C^{0,1}([0,T])$, then the assumption (A6) is valid.

\vspace{-4ex}
\paragraph{Section \ref{sec:4.2}}
\vspace{-2ex}

Theorem \ref{th:4.2} shows that for $\gb \in (0,1)$ the convergence rate is $O(1/n^\gb)$, which is better than the convergence rates $O(1/n^{\gb-\ga})$ and $O(1/n^{1-\ga})$ in Theorems \ref{th:3.4}
and \ref{th:3.5}. Hence, they are not sharp.

Theorems \ref{th:4.3} and \ref{th:4.4} demonstrate that the convergence rate can be arbitrary slow
if the smoothness of $Q$ is weaker and weaker.

Finally, Theorem \ref{th:4.5} shows that there is a bounded operator such that the Trotter product
formula does not converge in the operator norm. This makes clear that Theorem \ref{th:3.2}
becomes false if the condition that $A$ is a holomorphic generator is dropped. Indeed the operator
$D_0$ which plays the role of $A$ of Theorem \ref{th:3.2} is not a generator of holomorphic semigroup.


\begin{thebibliography}{10}

\bibitem{CacNeiZag2001}
V. Cachia, H. Neidhardt, and V.~A. Zagrebnov.
\newblock Accretive perturbations and error estimates for the {T}rotter product
  formula.
\newblock {\em Integral Equations Operator Theory}, 39(4):396--412, 2001.

\bibitem{CacNeiZag2002}
V, Cachia, H. Neidhardt, and V.~A. Zagrebnov.
\newblock Comments on the {T}rotter product formula error-bound estimates for
  nonself-adjoint semigroups.
\newblock {\em Integral Equations Operator Theory}, 42(4):425--448, 2002.

\bibitem{CachZag2001}
V. Cachia and V.~A. Zagrebnov.
\newblock Operator-norm approximation of semigroups by quasi-sectorial
  contractions.
\newblock {\em J. Funct. Anal.}, 180(1):176--194, 2001.

\bibitem{CachZag1999}
V. Cachia, V.A. Zagrebnov
\newblock Operator-norm convergence of the Trotter product formula
for sectorial generators.
\newblock \textit{Lett. Math. Phys.} 50:  203-211, 1999.

\bibitem{CachZag2001-b}
V. Cachia and V.~A. Zagrebnov.
\newblock Operator-norm convergence of the {T}rotter product formula for
  holomorphic semigroups.
\newblock {\em J. Operator Theory}, 46(1):199--213, 2001.

\bibitem{IchNeiZag2004}
T. Ichinose, H. Neidhardt, and V.~A. Zagrebnov.
\newblock Trotter-{K}ato product formula and fractional powers of self-adjoint
  generators.
\newblock {\em J. Funct. Anal.}, 207(1):33--57, 2004.

\bibitem{IchinoseTamura1998}
T. Ichinose and H. Tamura.
\newblock Error estimate in operator norm of exponential product formulas for
  propagators of parabolic evolution equations.
\newblock {\em Osaka J. Math.}, 35(4):751--770, 1998.

\bibitem{ITTZ2001}
T. Ichinose, Hideo Tamura, Hiroshi Tamura, and V.~A. Zagrebnov.
\newblock Note on the paper: ``{T}he norm convergence of the {T}rotter-{K}ato
  product formula with error bound'' by {T}. {I}chinose and {H}. {T}amura.
\newblock {\em Comm. Math. Phys.}, 221(3):499--510, 2001.

\bibitem{Kato1974}
T. Kato.
\newblock On the {T}rotter-{L}ie product formula.
\newblock {\em Proc. Japan Acad.}, 50:694--698, 1974.

\bibitem{Kato1978}
T. Kato.
\newblock Trotter's product formula for an arbitrary pair of self-adjoint
  contraction semigroups.
\newblock In {\em Topics in functional analysis (essays dedicated to {M}. {G}.
  {K}re\u\i n on the occasion of his 70th birthday)}, volume~3 of {\em Adv. in
  Math. Suppl. Stud.}, pages 185--195. Academic Press, New York-London, 1978.

\bibitem{Kato1980}
T. Kato.
\newblock {\em Perturbation theory for linear operators}.
\newblock Classics in Mathematics. Springer-Verlag, Berlin, 1995.
\newblock Reprint of the 1980 edition.

\bibitem{NeiSteZag2016}
H. Neidhardt, A. Stephan, and V.~A. Zagrebnov.
\newblock {Convergence rate estimates for Trotter product approximations of
  solution operators for non-autonomous Cauchy problems}.
\newblock {\em arXiv:1612.06147 [math.FA]}, December 2016.

\bibitem{NeiSteZag2017b}
H. Neidhardt, A. Stephan, and V.~A. Zagrebnov.
\newblock {Remarks on the operator-norm convergence of the Trotter product
  formula}.
\newblock {\em ArXiv 1703.09536 [math-ph]}, March 2017.

\bibitem{NeiSteZag2017}
H. Neidhardt, A. Stephan, and V.A. Zagrebnov.
\newblock {On convergence rate estimates for approximations of solution
  operators for linear non-autonomous evolution equations.}
\newblock {\em {Nanosyst., Phys. Chem. Math.}}, 8(2):202--215, 2017.

\bibitem{NeidhardtZagrebnov1998}
H.~Neidhardt and V.~A. Zagrebnov.
\newblock On error estimates for the {T}rotter-{K}ato product formula.
\newblock {\em Lett. Math. Phys.}, 44(3):169--186, 1998.

\bibitem{NeiZag1999-b}
H. Neidhardt and V.~A. Zagrebnov.
\newblock Trotter-{K}ato product formula and operator-norm convergence.
\newblock {\em Comm. Math. Phys.}, 205(1):129--159, 1999.

\bibitem{Nei1981}
H. Neidhardt.
\newblock On abstract linear evolution equations. {I}.
\newblock {\em Math. Nachr.}, 103:283--298, 1981.

\bibitem{Nei1981-b}
{{H. Neidhardt.
\newblock {On abstract linear evolution equations. II.}
\newblock {Prepr., Akad. Wiss. DDR, Inst. Math. P-MATH-07/81, 56 p. (1981).},
  1981.}}

\bibitem{Nei1982}
{{H. Neidhardt.
\newblock {On linear evolution equations. III: Hyperbolic case.}
\newblock {Prepr., Akad. Wiss. DDR, Inst. Math. p-MATH-05/82, 74 p. (1982).},
  1982.}}

\bibitem{NeiZag1999}
H. Neidhardt and V.~A. Zagrebnov.
\newblock Fractional powers of self-adjoint operators and {T}rotter-{K}ato
  product formula.
\newblock {\em Integral Equations Operator Theory}, 35(2):209--231, 1999.

\bibitem{NeiZag2009}
{{H. Neidhardt and V.~A. Zagrebnov.
\newblock Linear non-autonomous {C}auchy problems and evolution semigroups.
\newblock {\em Adv. Differential Equations}, 14(3-4):289--340, 2009.}}

\bibitem{Nickel1996}
{{G. Nickel.
\newblock On evolution semigroups and nonautonomous {C}auchy problems.
\newblock {\em Diss. Summ. Math.}, 1(1-2):195--202, 1996.}}

\bibitem{Rogava1991}
Dzh.~L. Rogava.
\newblock {{Error bounds for Trotter-type formulas for
self-adjoint operators, \textit{Funct. Anal. Appl.} 27: 217--219, 1993}}

\bibitem{Tam2000}
H.Tamura.
\newblock A remark on operator-norm convergence of {T}rotter-{K}ato product
  formula.
\newblock {\em Integral Equations Operator Theory}, 37(3):350--356, 2000.

\bibitem{Trotter1959}
H.~F. Trotter.
\newblock On the product of semi-groups of operators.
\newblock {\em Proc. Amer. Math. Soc.}, 10:545--551, 1959.

\end{thebibliography}

\def\cprime{$'$}

\end{document}